\documentclass[a4paper,12pt,oneside,reqno]{amsart}
\usepackage{amsmath,amsthm,amsfonts,amssymb,ifpdf}

\usepackage{amssymb}
\usepackage{amsmath}
\usepackage{amsthm}
\usepackage{graphicx}
\usepackage[all]{xy}
\usepackage{enumerate}
\usepackage{tikz-cd}
\usepackage{subcaption}

\ifpdf
  \usepackage[
    pdftex,
    colorlinks,%
    linkcolor=blue,citecolor=red,urlcolor=blue,
    hyperindex,%
    plainpages=false,%
    bookmarksopen,%
    bookmarksnumbered%
  ]{hyperref}
 \usepackage{thumbpdf}
\else
  \usepackage{hyperref}
\fi
\newtheorem{thm}{Theorem}[section]

\newtheorem{prop}{Proposition}[section]

\newtheorem{corr}{Corollary}[section]

\newtheorem{rmrk}{Remark}[section]

\newtheorem*{Proof*}{Proof}
\newtheorem*{Theorem*}{Theorem}
\begin{document}
\baselineskip=16pt

\setcounter{tocdepth}{1}

\providecommand{\keywords}[1]
{
  \small
  \textbf{\textit{Keywords---}} #1
}


\title[]{Universal Functions for the Sharkovsky classes of maps}

\author[V. Kannan]{V. Kannan}
\address{V. Kannan,
SRM University-AP, Amaravati-522502, India.}
\email{kannan.v@srmap.edu.in}

\author[Pabitra Narayan Mandal]{Pabitra Narayan Mandal}
\address{Pabitra Narayan Mandal, School of Mathematics and Statistics, University of Hyderabad, Hyderabad 500046, India.}
\email{pabitranarayanm@gmail.com}

\footnotetext{2020 Mathematics Subject Classification. Primary 37E15, 26A18; Secondary 37C25}

\begin{abstract}
We exhibit a single interval map (called universal map) that admits all those orbit patterns which are available in the first Sharkovsky class. An interval map is said to be in the first Sharkovsky class if every periodic point of it is a fixed point. This provides a way to find universal maps in the class of contractions on interval. We also characterize all such universal maps in the first Sharkovsky class.
\end{abstract}
\keywords{Orbit-pattern; universal function; contraction map; first Sharkovsky class.}
\maketitle

\section{Introduction and Preliminaries}
The function $\sin \frac{1}{x}$ (topologist's sine curve) and its sisters $x \sin \frac{1}{x}$ and $x^2 \sin \frac{1}{x}$ are encountered frequently in Real Analysis, to provide counter examples such as:
\begin{itemize}
\item a discontinuous function with connected graph,

\item a connected planar set that is not path-connected,

\item a discontinuity of second kind,

\item a continuous function that is not of bounded variation,

\item a non-rectifiable curve,

\item a differentiable function that is not continuously differentiable,

\item a continuous function that is not uniformly continuous, 

and so on.
\end{itemize}

In this paper, we study the function $x \sin \frac{1}{x}$ from the view of topological dynamics. While doing so, we find one glaring contrast. In Real Analysis, it had a negative role of being peculiar. But in Topological Dynamics, it plays a positive role of ``synthesizing" i.e., putting all things together. 

We are able to completely describe the order-patterns of all orbits for this map. As a prelude for this, we first observe that this map is in the first Sharkovsky class $\mathcal{F}_1$ and then describe a convenient index set for describing uncountably many order-patterns of trajectories available for maps in this class.

Our result yields as a by-product that (in a sense) this is a universal map for $\mathcal{F}_1$. We next characterize all universal maps for $\mathcal{F}_1$. There are uncountably many conjugacy types of these, but these maps together have a neat description. The Section 4 concludes with some equivalent ways of this description.

If $\mathcal{F}$ is a class of dynamical systems, an element $f\in \mathcal{F}$ is said to be universal for $\mathcal{F}$ if $f$ admits all orbit patterns that are available for maps in $\mathcal{F}$. Some families $\mathcal{F}$ admit universal elements and some others do not. For instance, if $\mathcal{F}$ is the class of $t$-simple systems on $\mathbb{R}$, there is no universal element in it (see \cite{kannan}). As a note, in topological dynamics a similar idea of universality with respect to $\omega$-limit sets appeared earlier. For details, see \cite{garc}, \cite{pokl}.

Two real sequences $(a_n)_{n=0}^{\infty}$ and $(b_n)_{n=0}^{\infty}$ are said to be of the same order-pattern if $a_m<a_n$ $\iff$ $b_m<b_n$ holds for all $m,n \in \mathbb{N}_0$. An order-pattern of a sequence $(f^n(x))_{n=0}^\infty$ in a real dynamical system $(\mathbb{R},f)$ or $(I,f)$ is called an orbit-pattern. Here $I$ is a closed interval in the real line and any continuous map from $I$ to $I$ is called an interval map. We say that an interval map $f$ admits an order-pattern $(a_n)_{n=0}^{\infty}$ if $\exists $ $x\in I$ such that $(f^n(x))_{n=0}^{\infty}$ and $(a_n)_{n=0}^{\infty}$ are having the same order pattern. 

We denote by $\mathcal{F}_1$ (known as maps of first-Sharkovskii type) the set of all interval maps that do not admit a $2$-cycle. Similarly for $\mathcal{F}_n$ is the set of all interval maps that do not admit a $2^n$-cycle. We say that an orbit pattern $(x_n)$ does not force a $2$-cycle if there exists an interval map $f$ admitting $x_n$ as a trajectory such that it does not admit a $2$-cycle (in other words, $f\in \mathcal{F}_1$). The next proposition gives equivalent description (without proof) for maps in $\mathcal{F}_1$.

\begin{prop}
The following are equivalent for an interval map $f$:

1. $f$ does not admit a $2$-cycle.

2. $f$ is an anti-symmetric relation.

3. Every periodic point of $f$ is a fixed point of $f$.

4. If $x$ is between $y$ and $f(y)$, then $y$ is not between $x$ and $f(x)$, unless they are equal. 
\end{prop}

An element $x$ is called a wall in its trajectory, if all the future terms are on the same side of it. In a sequence $(a_n)$ we note that every term is a wall, if for all $n\in \mathbb{N}$,

i) $a_{n+1}>a_n$ implies $a_m>a_n$ for all $m>n$,

ii) $a_{n+1}<a_n$ implies $a_m<a_n$ for all $m>n$,

iii) $a_n=a_{n+1}$ implies $a_{n}=a_m$ for all $m>n$.

The following result throws more light on the kind of sequences that we are studying in this paper. 

\begin{prop}\label{wall}
The following are equivalent for a (convergent) sequence $(a_n)$ in $I$, with distinct terms

1. Every term is a wall; that is, all future terms lie on one side of it. 

2. It is the union of an increasing sequence followed in its right by a decreasing sequence, where one of these two subsequences may be finite. 

3. Its terms go nearer and nearer to its limit p, when they are on the same side of p; i.e.,  $|a_m - p |$ $<$ $|a_n - p|$ for all those $m>n$ such that $a_m$ and $a_n$ are on the same side of $p$.

4. It (as an orbit pattern) does not force a 2-cycle. 
\end{prop}

The equivalence of $(1)$ and $(4)$ has been proved in \cite{block}, \cite{sharkov1}. Other implications among the above can be proved but we omit the proof. We also mention (without proof) a well known result which we are going to use in successive sections.
\begin{thm}(\cite{block}) \label{exact}
If $I$ and $J$ are closed intervals such that $f(I)$ contains $J$, then there is a closed subinterval $K$ of $I$ such that $f(K) = J$.
\end{thm}




The main theorems proved here are the following:

\begin{Theorem*} [A]
The set $\mathbb{P}$ of all orbit-patterns available in $\mathcal{F}_1$ can be naturally indexed by the set $J=\{L,R\}^\mathbb{N}\cup \{L,R\}^*$.
\end{Theorem*}

This is reminiscent of the theory of continued fractions where every real number has a tag from $\mathbb{N}^{\mathbb{N}}\cup \mathbb{N}^*$.

\begin{Theorem*} [B]
The function $f(x)=r x \sin \frac{1}{x}$ for some $0<r<1$ is an $\mathcal{F}_1$-map on $I=[-1,1]$ with the following universal property: If $g$ is any $\mathcal{F}_1$-map on $[-1,1]$, for all $p\in I$, $\exists q\in I$ such that $(g^n(p))_{n=0}^{\infty}$ and $(f^n(q))_{n=0}^{\infty}$ are of the same order-pattern.
\end{Theorem*}
 
\begin{Theorem*} [C]
The following are equivalent for $f\in \mathcal{F}_1$:

a) $f$ is universal in $\mathcal{F}_1$ i.e., all orbit-patterns that are available for any $g$ in $\mathcal{F}_1$ are available for this $f$.

b) $f$ admits a fixed point $p$ for which arbitrarily near $p$, on both sides of $p$, the values taken by $f$ swing both above $p$ and below $p$.
\end{Theorem*}

\section{Enlisting orbit-patterns in $\mathcal{F}_1$}
Let $(x_n)$ be such that $x_n=f(x_{n-1})$ where $n\geq 2$, for some $f\in \mathcal{F}_1$. Then $x_n\to p$ for some $p$ and $f(p)=p$. It is important to note that if some $x_k=p$, then $x_i=p$ for $i\geq k$.

Label every term of $(x_n)$ whenever $x_n \neq p$ with $L$ or $R$  according as it moves to its left or right. When $k$ is the least natural number such that $x_k=p$ then label upto $k-1$-th term of $(x_n)$ with $L$ or $R$ according as it moves to its left or right. Moreover it is important to note the following result.

\begin{prop}\label{sign}
Let $(x_n)$ be an orbit of an $\mathcal{F}_1$-map such that $x_n\to p$, then  
 
 $x_i$ is labeled as $L$ if and only if $x_i>p$ and 
 
 $x_i$ is labeled as $R$ if and only if $x_i<p$. 
\end{prop} 
\begin{proof}
This follows from Proposition \ref{wall}. Indeed, if $x_i$ is labeled as $L$ then $x_{i+1}<x_i$ and hence $x_n<x_i$ for all $n>i$. Therefore $p\leq x_i$. But $x_i\neq p$. Therefore $x_i>p$. 

Conversely, if possible, $x_i>p$ and $x_i$ is labeled as $R$. Then by our foregoing argument we get $x_n>x_i$ and $p\geq x_i$, which is a contradiction. 

Similarly for the second also.
\end{proof}

We choose the index set $J$ as the union of $\{L,R\}^\mathbb{N}$ and $\{L,R\}^*$. Here $\{L,R\}^\mathbb{N}$ denotes the set of all sequences over $\{L,R\}$ and $\{L,R\}^*$ denotes the set of all words over $\{L,R\}$. Let $\mathbb{P}$ denote the set of all orbit-patterns available in $\mathcal{F}_1$. 

In general, the sequence over $\{L,R\}$ need not determine the orbit pattern. In other words, two different patterns may have the same $L$-$R$-sequence. Here is an example, $((-1)^n\frac{1}{n})$ and $((-1)^nn)$ are having same $L$-$R$-sequence (namely $\overline{RL}$) but with different orbit patterns. This can be better understood as follows: The $L$-$R$-sequence contains only a part of the information that an order-pattern provides. The order-pattern provides the information: For each pair (m,n) of positive integers, which is smaller between the two numbers $x_m$ and $x_n$? The $L$-$R$-sequence provides a part of this information, namely when $n=m +1$. 

We do not expect this partial information to determine the full information. But it happens when we confine our discussion to the orbit patterns available in $\mathcal{F}_1$. To the question: given $m<n$, is $f^m(x)<f^n(x)$ or not?, the answer gets determined as follows: if the $(m+1)$-th term in the $L$-$R$-sequence is $R$ then $f^m(x)<f^n(x)$; if it is $L$, then $f^n(x)<f^m(x)$. Moreover the next theorem shows that $J$ is a natural index set for $\mathbb{P}$.

\begin{thm}\label{index}
There exists a natural bijection $\phi: \mathbb{P}\to J$.
\end{thm} 
\begin{proof}
Let us define $\phi: \mathbb{P}\to J$ by taking $\phi ((x_n))$ as that element of $J$ whose $i$-th term is the labeled of $x_i$ as defined earlier, $\forall i$. The function $\phi$ is well-defined. Indeed if $(a_n)$ and $(b_n)$ are representatives of two orbits have the same order pattern then $\phi$ takes them to the same sequence or word over $\{L,R\}$. This is because, the $n$-th term of $\phi((a_n))$ is $L$ iff $a_n > a_{n+1}$. This happens iff $b_n > b_{n+1}$. This happens if and only if the $n$-th term of $\phi((b_n))$ is also $L$.  It is important to note that the empty word corresponds to the constant orbit i.e., the orbit of a fixed point.

Let $w = w_1w_2w_3...\in J$. Take 

\[
x_1 = 
     \begin{cases}
       1 &\quad\text{if $w_1=L$}\\
       -1 &\quad\text{if $w_1=R$}\\
       0 &\quad\text{if $w_1$ does not exist.}\\
     \end{cases}
\] 

Suppose $x_1$, ... , $x_k$ have been defined. To define $x_{k+1}$:

\[
x_{k+1} = 
     \begin{cases}
       \frac{1}{k+1} &\quad\text{if $w_{k+1}=L$}\\
       -\frac{1}{k+1} &\quad\text{if $w_{k+1}=R$}\\
       0 &\quad\text{if $w_{k+1}$ is not there.}\\
     \end{cases}
\] 

The set $\{x_k:k\in \mathbb{N}\}$ constructed above is a discrete set. If $0\in \{x_k:k\in \mathbb{N}\}$, it is finite. If $0\notin \{x_k:k\in \mathbb{N}\}$ then $\{x_k: k \in \mathbb{N}\}\subset \{\pm \frac{1}{n}:n\in \mathbb{N}\}$. For every element (except at most two at the boundary) there is a next element and previous element. For every pair of adjacent elements say $x_m$ and $x_n$ we join the point $(x_n,x_{n+1})$ and $(x_m,x_{m+1})$ by a line segment in the plane. When this is done for all such pairs and $(0,0)$ is also taken, the graph of a function is ready; let the function be called $f$. 

By the construction of the map $f$, it is clear that $f$ is continuous. Observe that $|f(x)|<|x|$ on $[\text{inf}\; \{x_k\},\text{sup}\; \{x_k\}]$ except 0, where $f(0)=0$. If $p\neq 0$, then $|f^2(p)|<|p|$ i.e., $p$ can not be a periodic point of period $2$. Hence $f\in \mathcal{F}_1$ where $\phi ((x_n))=w$. Hence $\phi$ is surjection. 

Let $\alpha \in J$. Then $\phi ((x_n))=\alpha$ where $(x_n)$ is a representative of a orbit pattern in $\mathcal{F}_1$. For $m<n$ in $\mathbb{N}$, we can decide if $x_m<x_n$ or $x_m=x_n$ or $x_m>x_n$ by use of wall condition. Let $\alpha_m$ be the $m$-th term of $\alpha$. If $\alpha_m$ is $L$, then $x_m>x_n$. If $\alpha_m$ is $R$, then $x_m<x_n$. If $\alpha_m$ is empty, then $x_m=x_n$. Thus the orbit pattern of $(x_n)$ is decided by $\alpha$. This proves: $\phi$ is injection.
\end{proof}

\begin{rmrk}
Enlisting all orbit-patterns available for $\mathcal{F}_1$-maps, serves two other purposes (apart from the fact that it is natural on its own):

1. It paves the way for the proof of the later theorems of this paper.

2. It can be used to choose one representative from each orbit-pattern, from the set $\{\pm \frac{1}{n}: n\in \mathbb{N}\}\cup \{0\}$ as we have done in the proof of Theorem \ref{index}.
\end{rmrk}

\section{Motivating example of a universal function}
The dynamics of a system is said to be completely understood if we can describe all the orbit patterns in it. An excellent monograph about combinatorial orbit patterns for one dimensional maps is \cite{misi}.  In elementary text-books like \cite{holmgren}, examples have been provided, where such a complete understanding is possible. In all those examples, there are only finitely many, or at most countably many order patterns available. It will indeed be nice if the same is achieved for a more complicated example. This is what we do in this section. We are able to describe all the order-patterns available for the interval map $r x \sin \frac{1}{x}$ on $[-1,1]$ for $0<r\leq 1$. 


\begin{thm}\label{motivation}
Let $0<r<1$. Then the function $f(x):= r x \sin \frac{1}{x}$ on $[-1, 1]$ is universal. In other words, for any $\alpha \in J$, $\exists x_\alpha \in [-1, 1]$ whose orbit pattern is $\phi(\alpha)$.
\end{thm}  
\begin{proof}
It is easy to see that $f(x)=r x \sin \frac{1}{x}$ for $0<r<1$ is a first Sharkovskii's type map. Note that $0$ is the only fixed point. We will prove this theorem in three steps.

\noindent \textbf{Step-I}: Let $\delta > 0$. Then there exist four closed intervals $L_1$, $L_2$, $L_3$, $L_4$ such that:

$L_1$ and $L_2$ are inside the open interval $(0, \delta)$,

$L_3$ and $L_4$ are inside the open interval $(- \delta, 0)$,

$f(L_1)$ and $f(L_3)$ are of the form $[0, \eta]$,

$f(L_2)$ and $f(L_4)$ are of the form $[-\eta,0]$ for some $\eta>0$. Indeed for a given $\delta>0$ choose $a\in (0,\delta)$ such that $f(a)=0$ and $f$ is increasing at $a$. Choose $b\in (a,\delta)$ such that $f$ is increasing on $(a,b)$. Take $L_1=[a,b]$. Then $f(L_1)=[0,f(b)]$ as we want. Similarly $L_2$, $L_3$, $L_4$ are chosen with corresponding modifications. Observe that $L_1$, $L_2$, $L_3$ and $L_4$ do not contain $0$.

\noindent \textbf{Step-II}: Let $s=(s_n)$ be in $\{L,R\}^\mathbb{N}$. Define
\[   
I_1 = 
     \begin{cases}
       [\frac{1}{\pi}-\delta_1,\frac{1}{\pi}+\delta_1] &\quad\text{if $s_1=L$}\\
       [-\frac{1}{\pi}-\delta_1,-\frac{1}{\pi}+\delta_1] &\quad\text{if $s_1=R$}\\
     \end{cases}
\]
for a small enough $\delta_1>0$ such that $f$ is monotone on $I_1$. Observe that $f(I_1)$ is an interval containing $0$ as an interior point.

By Step-I, we choose a closed interval $J_1\subset f(I_1)$ not containing $0$ such that $f(J_1)$ is equal either $[0,\delta_2]$ or $[-\delta_2,0]$ for some $\delta_2>0$. In fact, if $s_2=L$, we choose $J_1\subset (0,\delta_1)\cap f(I_1)$, if $s_2=R$, we choose $J_1\subset (-\delta_1,0)\cap f(I_1)$.  By Theorem \ref{exact}, $\exists I_2\subset I_1$ such that $f(I_2)=J_1$. 

Suppose we have constructed $I_1\supset I_2\supset ... \supset I_k$ such that $f^j(I_j)$ is equal to $[0,\delta_{j}]$ or $[-\delta_{j}, 0]$ for some $\delta_j>0$ where $2\leq j\leq k$. To construct $I_{k+1}$, first we take a closed interval $J_{k}\subset f^k(I_k)$ not containing $0$ such that  $f(J_{k})$ is equal to either $[0,\delta_{k+1}]$ or $[-\delta_{k+1}, 0]$ for some $\delta_{k+1}>0$. In fact, if $s_{k+1}=L$, we choose $J_k\subset (0,\delta_{k})\cap f^k(I_k)$, if $s_{k+1}=R$, we choose $J_k\subset (-\delta_{k},0)\cap f^k(I_k)$. By Theorem \ref{exact}, $\exists I_{k+1}\subset I_k$ such that $f^k(I_{k+1})=J_k$.

Thus we have recursively constructed a nest of intervals $I_1\supset I_2\supset ... \supset I_n \supset ...$. Using NIT (Nested Interval Theorem), find an element $x$ common to all these $I_n$'s. Observe that for all $x\in I_1$, the orbit pattern-tag of $x$ starts with $s_1$. For all points $x\in I_2$, the orbit pattern-tag of $x$ starts with $s_1s_2$. More generally for all $n\in \mathbb{N}$, $\forall x\in I_n$, the orbit pattern-tag starts with $s_1s_2...s_n$. Here the orbit pattern tag of $x$ means $\phi (\alpha)$ where $\alpha$ is the orbit pattern of $(f^n(x))$. Therefore the orbit pattern tag of $x$ is nothing but the given sequence $s=(s_n)$.

\noindent \textbf{Step-III}: Let $w$ be a word of length $n$ over $\{L,R\}$. We will show that there exist an element $x_1$ in $[-1,1]$ such that $\phi ((x_n))=w$ where $x_n=f(x_{n-1})$ for $n\geq 2$. If $w$ is a word of length 1, then choose $x_1=\frac{1}{\pi}$ for $w=L$ or $x_1=-\frac{1}{\pi}$ for $w=R$. Assume $n\geq2$. Construct $J_1$, $J_2$,..., $J_{n-1}$ as in the previous. Note that $f(J_1)\supset J_2$, $f(J_2)\supset J_3$ ..., $f(J_{n-2})\supset J_{n-1}$. As $f(J_m)$ always contain $0$ for every $m$, choose $y_{n}\in J_{n-1}$ such that $f(y_{n})=0$. Choose $f(y_{j-1})=y_{j}$ and $y_j\in J_{j-1}$ for $2\leq j\leq n$. Again since $J_1\subset f(I_1)$, then $\exists y_1\in I_1$ such that $f(y_1)=y_2$. This $y_1$ has its orbit pattern represented by the given word $w$.
\end{proof}

\begin{rmrk}
One can prove that $x \sin \frac{1}{x}$ is also universal in $\mathcal{F}_1$ by a careful choice of $\delta_n's$ and $J_n's$ such that no $J_n$ contains a fixed point. Moreover $x sin \frac{1}{x}$ and $r x \sin \frac{1}{x}$ have the same collection of orbit patterns. One can use Theorem \ref{charact} to prove this (see Remark \ref{sin}).
\end{rmrk}

Let us denote by $\mathcal{C}$ the set of all contraction maps on $I$. We next prove that the interesting fact that $\mathcal{C}$ and $\mathcal{F}_1$ admit the same collection of orbit patterns. Moreover we have the following consequence.

\begin{corr}\label{cont}
There exists a universal function in the class $\mathcal{C}$.
\end{corr}
\begin{proof}
In this proof, with out loss of generality we assume the domain $I=[-1,1]$. Let $f\in \mathcal{C}$. If possible, let $f$ admit a $2$-cycle $\{p,q\}$ so that $f(p)=q$, $f(q)=p$. Then $|f(p)-f(q)|=|p-q|$, which is a contradiction to the fact that $f$ is contraction. Hence $f\in \mathcal{F}_1$. Hence $\{\text{All orbit patterns available in $\mathcal{C}$}\}$ $\subset$ $\{\text{All orbit patterns available for $\mathcal{F}_1$}\}$. Take $f(x)=\frac{x^3}{5}\sin \frac{1}{x}$ on $[-1,1]$. Since $|f'(x)|\leq \frac{4}{5}$, $f$ is contraction. A similar approach as Theorem \ref{motivation} can be used to prove that it contains all the orbit patterns available in $\mathcal{F}_1$ on $I$. Therefore $\frac{x^3}{5} \sin \frac{1}{x}$ is universal for $\mathcal{C}$.
\end{proof}

\begin{rmrk}
For any $0<r<1$, $r x \sin \frac{1}{x}$ is not a contraction, so directly we are unable to use it to prove Corollary \ref{cont}.
\end{rmrk}

\begin{rmrk}
Our motivating example in this section serves two other purposes:

1. In some sense $r x \sin \frac{1}{x}$ for some $0<r<1$ is the simplest example in $\mathcal{F}_1$ that is universal for $\mathcal{F}_1$.

2. The proof there gives a glimpse of the more  complicated proof of the next section.
\end{rmrk}
\section{Characterization of Universal functions}
Why should we give two proofs, one for the particular case of $r x \sin \frac{1}{x}$ for some $0<r<1$ and again for the general case? Can we not omit the proof in Section 3? We have included both because: 
The first proof gives the basic idea in the simplest case. It is easier. It needs refinement in the general proof. There are two main difficulties in the general proof (that were not encountered in the earlier proof). First, the set of points that go to the fixed point can have several limit points; that is, infinite fluctuations can happen at several points. Second, the graph of the function may cross the diagonal often, spoiling our easier method. This necessitates more care on the choice of $J_n$'s. The end point of $J_n$ that is in the pre-image of the fixed point, is chosen carefully so that it is isolated from the required side. (In the notation of the proof, there is no point strictly between $b_n$ and $a_n$ that goes to $p$.) 
Next, for all points in $J_n$, the motion under $f$ is unilateral. This is guaranteed by keeping $J_n$ inside $\mathcal{S}$ (defined in the proof). These requirements are not needed in the simpler case; thus the particular proof also deserves to be grasped separately.

\begin{thm}\label{charact}
A function $f\in \mathcal{F}_1$ is universal if and only if $\exists$ a sequence $(a_n)$ converging to a fixed point $p$ such that 

$a_n>p$ if $n$ is even, 

$a_n<p$ if $n$ is odd, 

$f(a_n)>p$ if $n$ or $n-1$ is multiple of four  and 

$f(a_n)<p$ if $n-2$ or $n-3$ is multiple of four.
\end{thm}

\begin{proof}
First we will prove the reverse implication i.e., the conditions stated in the theorem is sufficient for a function to be universal. This part of the proof has two parts namely the preparatory part and the recursive definition part. 

In the preparatory part, first we define $b_n$ and and study some essential properties of $b_n$ (as Fact-I and Fact-II) which are useful in our proof. 

Let us define $b_n$ by
\[ b_n:= 
     \begin{cases}
       \text{sup} \; \{x: f(x)=p \; \text{and} \; x<a_n\} &\quad\text{where $a_n>p$}\\
       \text{inf}\; \{x: f(x)=p \; \text{and} \; x>a_n\} &\quad\text{where $a_n<p$}\\
     \end{cases}
\]

\noindent \textbf{Fact-I}: No other points between $b_n$ and $a_n$ go to $p$. Here $f(b_n)=p$ but no $b_n$ is equal to $p$. And both $a_n\to p$, $b_n\to p$. Therefore there is an interval with $b_n$ as an end point whose image is an interval with $p$ as an end point. Moreover for any $\delta>0$, $\exists$ $n\in \mathbb{N}$ and $0<\epsilon <\delta$ such that

$b_{4n} \in [p,p+\delta]$ and $f(L^n_0)\supset [p,p+\eta]$ where $L^n_0=[b_{4n},b_{4n}+\epsilon]$; 

$b_{4n+1}\in [p-\delta,p]$ and $f(L^n_1)\supset [p,p+\eta]$ where $L^n_1=[b_{4n+1}-\epsilon,b_{4n+1}]$; 

$b_{4n+2}\in [p,p+\delta]$ and $f(L^n_2)\supset [p-\eta,p]$ where $L^n_2=[b_{4n+2},b_{4n+2}+\epsilon]$; 

$b_{4n+3}\in [p-\delta,p]$ and $f(L^n_3)\supset [p-\eta, p]$ where $L^n_3=[b_{4n+3}-\epsilon,b_{4n+3}]$ 
for some $\eta>0$. 

\noindent \textbf{Fact-II}: Define $\mathcal{S}:=\{x\in I: |f(x)-p|<|x-p|\}$. Observe that $b_n\in \mathcal{S}$, $\forall n\in \mathbb{N}$. Moreover each $b_n$ is an interior point of $\mathcal{S}$. Therefore for suitable $\eta>0$, we can choose $L^n_0,L^n_1,L^n_2,L^n_3$ contained in $\mathcal{S}$ such that Fact-I holds.

Let $s=(s_n)$ be any sequence over $\{L,R\}$. We will show that $\exists$ some element $x$ in the domain of $f$ such that orbit pattern tag of $x$ is $s$.

In the recursive definition part, as a base step first observe that $a_1<p$ and $f(a_1)>p$, $a_2>p$ and $f(a_2)<p$, $a_3<p$ and $f(a_3)<p$, $a_4>p$ and $f(a_4)>p$.

If $s_1=L$ we will search for an element in the right side of $p$ and if $s_1=R$ we will search for an element in the left side of $p$ (because of Proposition \ref{sign}). We choose $I_1$ such that 
\[   
I_1 = 
     \begin{cases}
       [b_{2}, b_{2}+\delta_1] \; \text{or} \; [b_4,b_4+\delta_1]  &\quad\text{if $s_1=L$}\\
       [b_{1}-\delta_1,b_{1}] \; \text{or} \; [b_3-\delta_1,b_3]  &\quad\text{if $s_1=R$}\\
     \end{cases}
\]
for a small enough $\delta_1>0$ so that $f(I_1)$ is an interval with $p$ as an end point and $|f(x)-p|<|x-p|$ for all $x\in I_1$. Note that whether we will chose $b_2$ or $b_4$ or $b_1$ or $b_3$, that will be decided by $s_2$.  

Choose $n_1$ large enough such that $|a_n-p|<\delta_1$ for all $n\geq n_1$. Then choose $m_1>n_1$ such that 

$m_1$ is a multiple of $4$ if $s_2=L=s_3$, 

$m_1$ is of the form $4k+1$ if $s_2=R$ and $s_3=L$, 

$m_1$ is of the form $4k+2$ if $s_2=L$ and $s_3=R$, 

$m_1$ is of the form $4k+3$ if $s_2=R=s_3$. 

By Fact-I and Fact-II we can choose an interval $J_1\subset f(I_1)$ with $b_{m_1}$ as an end point such that $f(J_1)$ is equal to $[p-\delta_2,p]$ or $[p,p+\delta_2]$ for some $\delta_2>0$ and $|f(x)-p|<|x-p|$ for $x\in J_1$. In fact, $J_1\subset (p,p+\delta_1)\cap f(I_1)$ if $s_2=L$ or $J_1\subset (p-\delta_1,p)\cap f(I_1)$ if $s_2=R$. By Theorem \ref{exact}, $\exists I_2\subset I_1$ such that $f(I_2)=J_1$.

Suppose we have constructed $I_1\supset I_2\supset ... \supset I_k$ such that $f^j(I_j)$ is equal to $[p,p+\delta_{j}]$ or $[p-\delta_{j},p] $ for some $\delta_j>0$ where $2\leq j \leq k$. 

To construct $I_{k+1}$: We make a succession of choices as under:

First choose $n_k\in \mathbb{N}$ such that $|a_{n}-p|<\delta_k$ for all $n\geq n_k$.

Next look at $s_{k+1}$ (which may be $L$ or $R$), accordingly choose $m_k>n_k$ so that $b_{m_k}$ is as wanted (After choosing $n_k$, we choose $m_k$. Its parity or rather its conjugacy class modulo 4, is determined by $s_{k+1}$ and $s_{k+2}$.) and choose $J_k\subset f^k(I_k)$ with $b_{m_k}$ as an end point such that
$f(J_{k})$ is equal to $[p,p+\delta_{k+1}]$ or $[p-\delta_{k+1}, p]$ for some $\delta_{k+1}>0$ where $|f(x)-p|<|x-p|$ for all $x\in J_k$. In fact, if $s_{k+1}=L$, we choose $J_k\subset (p,p+\delta_k)\cap f^k(I_k)$, if $s_{k+1}=R$, we choose $J_k\subset (p-\delta_k,p)\cap f^k(I_k)$. 

Lastly take $I_{k+1}\subset I_k$ such that $f^k(I_{k+1})=J_{k}$ (by Theorem \ref{exact}).

Thus we have recursively constructed a nest of intervals $I_1\supset I_2\supset ... \supset I_n \supset ...$. Using NIT, we find that $\cap_{n=1}^{\infty}I_n\neq \emptyset$. By analogous method (used in Theorem \ref{motivation}), the orbit-pattern-tag of $x\in \cap_{n=1}^{\infty}I_n$ is nothing but the given sequence $s=(s_n)$.

Let $w$ be a word of length $n$ over $\{L,R\}$. We will show that there exists an element $x_1$  such that $\phi ((x_n))=w$ where $x_n=f(x_{n-1})$ for $n\geq 2$. As above, construct an interval $I_w$ such that every element in it has its orbit tag starting with $w$. Lastly since $f(I_w)=J_{n-1}$, there is $x$ in $I_w$ whose image in $J_n$ goes to $p$. The orbit pattern tag of this $x$ is $w$.

Conversely, since $f\in \mathcal{F}_1$ is universal, there exists $x$ in $I$ with the orbit pattern $\overline{RRLL}$. Let $p$ be the limit point of $(f^n(x))$. Take $a_{4k+3}=f^{4k}(x)$, $a_{4k+1}=f^{4k+1}(x)$, $a_{4k}=f^{4k+2}(x)$ and $a_{4k+2}=f^{4k+3}(x)$. The sequence $(a_n)$ is the desired sequence.
\end{proof}

\begin{rmrk}\label{sin}
To see that Theorem \ref{motivation} is a particular case of Theorem \ref{charact}, we take 
\[   
a_n = 
     \begin{cases}
       \frac{1}{2n\pi + \frac{\pi}{2}} &\quad\text{if $n$ is of the form $4k$},\\
       \frac{1}{2n\pi - \frac{\pi}{2}} &\quad\text{if $n$ is of the form $4k+2$},\\
       -\frac{1}{2n\pi + \frac{\pi}{2}} &\quad\text{if $n$ is of the form $4k+3$},\\
       -\frac{1}{2n\pi - \frac{\pi}{2}} &\quad\text{if $n$ is of the form $4k+1$}.\\
     \end{cases}
\]
Then $a_n\to 0$ and $a_n \sin \frac{1}{a_n}$ is $>0$ if $n$ or $n-1$ is multiple of four and $<0$ if $n-2$ or $n-3$ is multiple of four. In particular this proves that $x \sin \frac{1}{x}$ is also universal in $\mathcal{F}_1$. 
\end{rmrk}

\begin{rmrk}
Our characterization of universal functions serves two other purposes also.

1. It throws further light on the forcing relation on the orbit-patterns. It is no longer a partial order (though it was partial order in some other instances \cite{bald}, \cite{kannan}).

2. It shows that polynomial interval maps in $\mathcal{F}_1$ admit only countably many order-patterns. 
\end{rmrk}
\begin{prop}
Let $f$ in $\mathcal{F}_1$ admit a fixed point $p$, say. There are five more equivalent conditions for $f$ to be universal: 

i) For every $\epsilon > 0$, there is $\delta > 0$ such that both $f ( [ p,p+ \epsilon])$ and $f([p - \epsilon, p])$ contain $[p-\delta, p+\delta]$. 

ii) Arbitrarily near $p$, on either side of $p$, both values $> p$ and values $< p$ are taken. 

iii) For every $r>0$, $inf \; f([p,p+r])<p$, $sup \; f([p,p+r])>p$; $inf \; f([p-r,p])<p$ and $sup \; f([p-r,p])>p$. 

iv) $p$ is in the closure of these four sets namely, $f^{-1}((p,1])\cap (p,1]$, $f^{-1}((p,1]) \cap [0,p)$, $f^{-1}([0,p))\cap [0,p)$, $f^{-1}([0,p))\cap (p,1]$. 

v) $p$ is a limit point of (each of) some four sequences $(a_n)$, $(b_n)$, $(c_n)$ and $(d_n)$, where for every $n$, the four numbers $a_n$, $c_n$, $f(a_n)$ and $f(b_n)$ are $> p$ and the other four numbers $b_n$, $d_n$, $f(c_n)$ and $f(d_n)$ are $<p$. 
\end{prop}

\section*{Acknowledgement(s)}
The second author acknowledges NBHM-DAE (Government of India) for financial
support (Ref. No. 2/39(2)/2016/NBHM/R \& D-II/11397). This work was done during the second author's visit to SRM University-AP.


\begin{thebibliography}{1}
    \bibitem{misi} Alseda, L., Llibre, J., Misiurewicz, M., {\em Combinatorial dynamics and entropy in dimension one}, Second edition. Advanced Series in Nonlinear Dynamics, 5. World Scientific Publishing Co., Inc., River Edge, NJ, 2000.
    
    \bibitem{bald} Baldwin, S., {\em Generalizations of a theorem of Sarkovskii on orbits of continuous real-valued functions}, Discrete Mathematics, 67, 111-127, North-Holland, 1987.
         
    \bibitem{block} Block, L. S., Coppel, W. A., {\em Dynamics in One Dimension}, Lecture Notes in Mathematics, Springer-Verlag Berlin Heidelberg, 1992.
	
	\bibitem{garc} Chudziak, J., Garcia Guirao, J. L., Snoha, L., Spitalsky, V., {\em Universality with respect to $\omega$-limit sets}, Nonlinear Anal. 71, 1485–1495, 2009.
	
    \bibitem{holmgren} Holmgren, R.,{\em A First Course in Discrete Dynamical Systems}, Springer-Verlag New York, 1996.
    	
	\bibitem{kannan} Kannan, V., Mandal, P. N., {\em Which orbit types force only finitely many orbit types?}, Journal of Difference Equations and Applications, 26, 676–692, 2020.
	
	\bibitem{Li} Li, T. Y., Misiurewicz, M., Pianigiani, G., Yorke, J. A., {\em No division implies chaos}, Trans. Amer. Math. Soc., 273, 191--199, 1982.
	
	\bibitem{pokl} Pokluda, D., Smital, J., {\em A ``universal" dynamical system generated by a continuous map of the interval}, Proc. Amer. Math. Soc., 128, 3047–3056, 2000.
	
	\bibitem{sharkov1} Sharkovsky, O. M., {\em On cycles and the structure of a continuous mapping}, (Russian) Ukrain. Mat. Z. 17, 104–111, 1965.

    
    
    
	
	
\end{thebibliography}
\end{document}